\documentclass[11pt, reqno, a4paper]{amsart}
\usepackage[
    textwidth=6.2in,
    textheight=9in,
    centering
]{geometry}
\usepackage[T1]{fontenc}
\usepackage{lmodern}
\usepackage{microtype}
\usepackage{amsmath,amssymb,amsfonts,amsthm}
\usepackage{mathtools} 
\usepackage{mathrsfs}
\usepackage{graphicx}
\usepackage{enumitem}
\usepackage{comment}

\usepackage{hyperref}
\usepackage{cleveref}
\numberwithin{equation}{section}
\allowdisplaybreaks

\newtheorem{innercustomthm}{Theorem}

\newenvironment{customthm}[1]
{%
  \renewcommand\theinnercustomthm{#1}%
  \begin{innercustomthm}
}
{%
  \end{innercustomthm}
}

\theoremstyle{definition}
\newtheorem{definition}{Definition}[section]
\newtheorem{example}[definition]{Example}

\theoremstyle{remark}
\newtheorem{remark}[definition]{Remark}

\theoremstyle{plain}
\newtheorem{theorem}[definition]{Theorem}
\newtheorem{result}[definition]{Result}
\newtheorem{lemma}[definition]{Lemma}
\newtheorem{proposition}[definition]{Proposition}
\newtheorem{corollary}[definition]{Corollary}

\newtheorem{question}[definition]{Question}

\usepackage{color}

\newcommand{\harm}{\omega({\gamma}, K)}
\newcommand{\harmz}{\omega(0,K)}
\newcommand{\harminfty}{\omega(\infty, K)}
\newcommand{\harma}{\omega(a, K)}
\newcommand{\harmad}{\omega(a, {\Omega})}
\newcommand{\harmy}{\omega(y, K)}
\newcommand{\harmyi}{\omega(y_i, K)}
\newcommand{\harmF}{\omega({\gamma}, F)}

\newcommand{\capa}{c}
\newcommand{\equi}{\nu_K}
\newcommand{\conj}{\mathcal O}
\newcommand{\res}{\operatorname{Res}}
\newcommand{\mut}{\mu_{\textbf{\textit{t}}}}
\newcommand{\minalp}{\mathcal M_{\alpha, K}}

\newcommand{\sph}{\widehat{\mathbb{C}}}
\newcommand{\C}{\mathbb{C}} 
\newcommand{\R}{\mathbb{R}}
\newcommand{\Z}{\mathbb{Z}}
\newcommand{\N}{\mathbb{N}}
\newcommand{\Q}{\mathbb{Q}}

\begin{document}

\title[Arithmetic probability measures]{Arithmetic probability measures}

\author{Mayuresh Londhe}
\address{Faculty of Engineering and Natural Sciences, Sabancı University, İstanbul, Turkey}
\email{mayuresh.londhe@sabanciuniv.edu, mayureshwar99@gmail.com}

\begin{abstract}
We consider probability measures on compact subsets of the complex plane 
arising as limiting distributions of Galois conjugates of certain algebraic integers. 
These measures are characterized by infinitely many integral inequalities, 
one for each nonzero integer polynomial.
We study a broad family of measures given by particular convex combinations of harmonic and equilibrium measures.
We then show that, for measures in this family, it suffices to verify two 
conditions to guarantee all the required inequalities.
Under additional arithmetic constraints on the underlying compact sets, we further show 
that the infinite collection of inequalities is equivalent to checking 
the inequalities associated with two explicitly specified integer polynomials.
\end{abstract}


\maketitle

\section{Introduction}\label{S:intro}
This paper studies probability measures of arithmetic origin
on the complex plane $\C$.
Given an algebraic number $\alpha$, let $\deg \alpha$ denote
the degree of its minimal polynomial over $\Q$, and let
$\conj(\alpha)$ denote the set of roots of this polynomial in $\C$, 
which we refer to as the Galois conjugates of $\alpha$.
Let $K \subset \C$ be a compact set, and 
denote by $\mathcal{P}(K)$ the set of Borel probability measures whose supports are contained in $K$.

\begin{definition}\label{D:APM}
A measure $\mu \in \mathcal{P}(K)$ is called an \emph{arithmetic probability measure}
if there exists a sequence $\{\alpha_n\}$ of
algebraic {\bf integers} satisfying the following conditions:
\begin{enumerate}[label=(\roman*)]
\item \label{I:deg} $\deg \alpha_n \to \infty$ as $n \to \infty$;
\item \label{I:neighbour} For every neighborhood $U$ of $K$, there exists $N \in \N$ such that
$\conj(\alpha_n)\subset U$ for all $n\geq N$;
\item \label{I:weak} The normalized counting measure on $\conj(\alpha_n)$ converges to $\mu$ 
in the weak* topology as $n \to \infty$, i.e.,
for every bounded continuous function $f: \C \to \R$,
\[
\lim_{n \to \infty} \frac{1}{\deg \alpha_n} \sum_{\gamma \in \conj(\alpha_n)} f(\gamma)= \int f(z) \, d \mu (z).
\]
\end{enumerate}
\end{definition}
\noindent{Since} $\conj(\alpha)$ is invariant under complex conjugation for every algebraic number $\alpha$,
arithmetic probability measures are symmetric with respect to the real axis.
We therefore restrict our attention to compact sets $K$ that are symmetric with respect to the real axis.
\smallskip

Potential theory provides a natural framework for studying arithmetic probability measures.
Classical results of Fekete and Szeg\H{o} \cite{Fekete:udvwbgmg23, FekeSzego:aeicwr55} 
show that a sequence of algebraic integers satisfying conditions \ref{I:deg} 
and \ref{I:neighbour} exists if and only if 
the logarithmic capacity $\capa(K)$ of $K$ satisfies $\capa(K) \geq 1$.
Logarithmic capacity is an analytic notion of size for compact sets; for example, 
a circle of radius $r$ has capacity $r$, while a line segment of length $l$ has capacity $l/4$.
We defer the relevant potential-theoretic definitions to Section~\ref{S:prelim}.
In particular, if $\capa(K) < 1$, then $K$ admits no arithmetic probability measures. 
If $\capa(K) =1$, there is exactly one arithmetic probability measure, namely the equilibrium measure of $K$
(see Remark~\ref{rem:negativepotential}).
When $\capa(K) > 1$, the situation is markedly different: 
$K$ admits infinitely many arithmetic probability measures.
Indeed, the equilibrium measures of symmetric compact subsets of $K$ with logarithmic capacity 
one already yield infinitely many such measures.
\smallskip

Although the logarithmic capacity of $K$ determines whether arithmetic probability measures exist, 
characterizing these measures presents a subtler problem.
Until recently, a complete characterization remained unknown.
We refer the reader to \cite{Serre19:dadvpef, T:serre19, NT:cdmcav25} 
for discussions of this problem and its various applications.
A major breakthrough was achieved by Smith \cite{Smith:aicpd24}, who characterized these measures for 
compact subsets of the real line with at most countably many connected components. 
Shortly thereafter, Orloski and Sardari \cite{OrlSar:ldcai23} extended this characterization to arbitrary 
compact subsets of $\C$.

\begin{result}[\cite{Smith:aicpd24, OrlSar:ldcai23}]\label{R:APM_char}
Let $\mu \in \mathcal P(K)$ be symmetric with respect to the real axis. Then
$\mu$ is an arithmetic probability measure if and only if 
\begin{equation}\label{E:infinite}
\int \log |Q(z)|\, d\mu(z) \geq 0 \qquad \forall\, Q \in \Z[z] \setminus \{0\},
\end{equation}
where $\Z[z]$ denotes the set of polynomials with integer coefficients.
\end{result}

As observed by Serre \cite{Serre19:dadvpef}, the forward implication of Result~\ref{R:APM_char} 
is straightforward. The converse direction, however, is substantially more involved. Its proof requires a delicate 
combination of potential-theoretic arguments and existence results from the geometry of numbers. 
While Result~\ref{R:APM_char} provides a complete characterization,
checking whether a given probability measure satisfies the inequalities in \eqref{E:infinite} requires verifying 
infinitely many inequalities.
As a result, it is not well understood which measures, even on circles and intervals, satisfy all these 
inequalities simultaneously. This motivates the following more general question, attributed to Peter Sarnak 
in \cite[page~2]{OrlSar:qcft23}.
\begin{question}[Sarnak]\label{Q:Sarnak}
What are the measures with compact support $K$ that satisfy the infinite number of inequalities
in \eqref{E:infinite}?
\end{question}

We consider a broad family of probability measures 
for which the infinitely many inequalities in \eqref{E:infinite} admit a finite criterion.
Let $K\subset\mathbb C$ be a compact set with $\capa(K)>0$ that is symmetric with respect to the real axis.
Suppose that $\alpha$ is an algebraic number whose Galois conjugates lie outside $K$, i.e.,
\[
\conj(\alpha) \cap K = \emptyset.
\]
For each $t\in[0,1]$, define the probability measure
\begin{equation}\label{E:conv_measure}
\mu_t(\alpha, K):= t{\equi} + \frac{1-t}{\deg \alpha} \sum_{\gamma \in \conj(\alpha)} {\harm},
\end{equation}
where $\nu_K$ is the equilibrium measure of $K$ and $\harm$ 
is the harmonic measure on $K$ with respect to $\gamma$
(see Section~\ref{S:prelim} for precise definitions).
The assumption $\conj(\alpha) \cap K = \emptyset$ ensures that each harmonic measure $\harm$ 
is well defined. By construction, $\mu_t(\alpha, K) \in \mathcal P(K)$ and is symmetric with respect to the 
real axis for every $t\in[0,1]$.
\smallskip

The family of measures $\mu_t(\alpha, K)$ is motivated, in part, by the special case 
$\alpha=0$, for which \eqref{E:conv_measure} reduces to
\begin{equation}\label{E:zero}
\mu_t(0, K)=t{\equi} + (1-t) \harmz.
\end{equation}
An arithmetic probability measure of this form supported on a certain interval was used by Smith \cite{Smith:aicpd24}
to give a negative answer to the classical Schur--Siegel--Smyth trace problem.
More broadly, the relevance of this family is underscored by the fact that arithmetic 
probability measures of the form \eqref{E:conv_measure} are dense
 in the space of all arithmetic probability measures with respect to the weak* topology (see Proposition~\ref{P:dense}).
\smallskip

Our first main result shows that, for measures of the form \eqref{E:conv_measure}, the infinite system of 
inequalities in \eqref{E:infinite} is controlled by two explicit conditions.
To state the result, for an algebraic number $\alpha$, let
$P_{\alpha}\in\Z[z]$ denote the unique primitive irreducible polynomial over $\Q$ with positive leading 
coefficient satisfying $P_{\alpha}(\alpha)=0$.
We also introduce
\[
\minalp:=\min_{z \in K} |P_{\alpha}(z)|
\]
and write $\log^+z:= \max\{\log z, 0\}$ for $z>0$.

\begin{customthm}{A}\label{ct:suff}
With $K$ and $\alpha$ as above, for $t \in [0,1]$, set $\mu_t:= \mu_t(\alpha, K)$.
Suppose that
\begin{align}
\int \log |P_\alpha(z)|\,d\mu_t(z)&\ge0, \notag \\
\frac{1-t}{\deg\alpha}
 \left( \int \log |P_\alpha(z)|\,d{\equi}(z)
-\log^+\minalp \right)
&\le \log \capa(K). \notag
\end{align}
Then
$\mu_t$ satisfies all the inequalities in \eqref{E:infinite}
and, by Result~\ref{R:APM_char}, is therefore an arithmetic probability measure.
\end{customthm}
The two conditions in Theorem~\ref{ct:suff} reflect the basic dichotomy according to 
whether an integer polynomial vanishes at $\alpha$.
The first controls polynomials vanishing at $\alpha$ through the inequality for $P_{\alpha}$, 
while the second provides the estimate needed for polynomials $Q\in\Z[z]$ 
with $Q(\alpha) \neq 0$.
Thus, for the measures $\mu_t(\alpha, K)$, it suffices to verify these two conditions in place of 
the infinitely many inequalities in \eqref{E:infinite}.
This criterion is non-vacuous whenever $\capa(K)>1$.
Indeed, a standard property of the equilibrium measure gives 
\[
\int \log |P_{\alpha}(z)| \, d {\equi}(z) \geq \log \capa(K) >0.
\]
It follows that both conditions are satisfied for $t$ sufficiently close to 1. 
Moreover, if the two conditions hold for some $t_0\in[0,1]$, then they hold for every $t\in[t_0,1]$.
\smallskip

The integral appearing in the second condition reflects the position of the Galois conjugates of $\alpha$
relative to $K$. To see this concretely and obtain an explicit formula, consider the circle
$K=\{z \in \C:|z|=r\}$. In this case,
\[
\frac{1}{\deg\alpha} \int \log |P_\alpha(z)|\,d{\equi}(z)
=\log r + \frac{1}{\deg\alpha} \bigg( \log \ell_{\alpha} 
+ \sum_{\gamma \in \conj(\alpha)} \log^+ \frac{|\gamma|}{r} \bigg),
\]
where $\ell_{\alpha}$ is the leading coefficient of $P_\alpha$.
When $\capa(K)=r>1$, the formula shows that, for algebraic integers, the closer the conjugates 
lie to the circle from the exterior, the smaller the threshold value of $t$ for satisfying the second 
condition. Moreover, this threshold tends to $0$ as the conjugates approach the circle
(see Remark~\ref{rem:smallt} for the general case).
This integral may be viewed as a $K$-relative logarithmic Mahler measure, or height, 
associated with $\alpha$.
For the unit circle, this is precisely the classical logarithmic Weil height $h(\alpha)$.
\smallskip

Theorem~\ref{ct:suff} also shows that sufficiently distant Galois conjugates automatically yield 
arithmetic probability measures.

\begin{corollary}\label{cor:faraway}
Assume that $\capa(K) > 1$. Then there 
exists a constant $r > 0$, depending only on $K$, such that if
every $\gamma \in \conj(\alpha)$ satisfies $|\gamma| \ge r$,
then $\mu_t(\alpha,K)$ is an arithmetic probability measure for all $t \in [0,1]$.
\end{corollary}

Under additional arithmetic assumptions on $K$,
we obtain two complete characterizations for the measures under consideration, 
presented in the next two theorems. In each case, the infinitely many inequalities in \eqref{E:infinite}
reduce to two inequalities involving explicitly specified integer polynomials.

\begin{customthm}{B}\label{ct:ns}
Assume every connected component of K is nonsingleton and that there exists an algebraic integer 
$\beta$ such that $P_{\alpha}(\beta)$ is an algebraic unit and $\conj(\beta) \subset K$.
Then $\mu_t:=\mu_t(\alpha,K)$ satisfies all the inequalities in \eqref{E:infinite}
if and only if
\[\int \log |P_{\alpha}(z)| \, d \mu_t(z) \geq 0 \quad {\text{and}} \quad
\int \log |P_{\beta}(z)| \, d \mu_t(z) \geq 0.\]
\end{customthm}
Recall that an algebraic unit is an algebraic integer whose inverse is also an algebraic integer.
As a simple example illustrating the preceding characterization, consider the special case $\alpha=0$.
Assume that $K \subset \R \setminus \{0\}$ is compact, every connected component of $K$ is 
nonsingleton, and $1 \in K$. 
Then $\mu_t:=\mu_t(0,K)$, which takes the form \eqref{E:zero},
is an arithmetic probability measure if and only if 
\[\int \log |z| \, d \mu_t(z) \geq 0 \quad {\text{and}} \quad
\int \log |z-1| \, d \mu_t(z) \geq 0.\]
Example~\ref{Ex:both_condi} shows, moreover, that neither of these two conditions, considered 
individually, is sufficient to characterize arithmetic probability measures in this setting. 
Theorem~\ref{ct:ns} is a special case of the more general result presented in Theorem~\ref{T:ns}.
\smallskip

We next turn to closed intervals, where the relevant potential-theoretic quantities admit explicit formulas. 
For intervals with integer endpoints, we obtain another complete characterization in terms of the two 
endpoint conditions below, which replace the infinitely many inequalities in \eqref{E:infinite}.

\begin{customthm}{C}\label{ct:interval}
Let $K=[a, b]$ be a closed interval with $a, b \in \Z$ and $a<b$, and
let $\alpha$ be an algebraic number such that $\conj(\alpha) \subset \R\setminus [a,b].$
Then $\mu_t:= \mu_t(\alpha,K)$ satisfies all the inequalities in \eqref{E:infinite}
if and only if
\[\int \log |z-a| \, d \mu_t(z) \geq 0 \quad {\text{and}} \quad
\int \log |z-b| \, d \mu_t(z) \geq 0.\]
\end{customthm}

In this setting, the measures appearing in the definition of
$\mu_t(\alpha,K)$ have explicit densities. 
For $K=[a,b]$, $a<b$,
the equilibrium measure is given by
\[
d\nu_K(x) = \frac{dx}{\pi\sqrt{(x-a)(b-x)}},
\qquad a<x<b.
\]
The harmonic measure with pole at $y\notin[a,b]$ is given by
\[
d\harmy(x)= \frac{\sqrt{|(y-a)(y-b)|}}{\pi |y-x|\sqrt{(x-a)(b-x)}}\,dx,
\qquad a<x<b.
\]
These formulas, combined with our results,
facilitate the construction of concrete examples of arithmetic probability measures.
\smallskip

Working with intervals with integer endpoints in fact allows us to consider a broader class of 
measures, as in Theorem~\ref{T:intval_z}, of which Theorem~\ref{ct:interval} is a special case. 
In this broader setting, the measures are defined using arbitrary convex combinations of equilibrium 
and harmonic measures, with harmonic measures taken with respect to arbitrary real points. 
A corresponding characterization can also be obtained for finite unions of intervals with integer endpoints.
Finally, the integral conditions in Theorem~\ref{ct:interval} remain sufficient even for arbitrary real 
endpoints $a, b$, although they are no longer necessary in general. 
\smallskip

The paper is organized as follows. Section~\ref{S:prelim} develops the potential-theoretic background 
used throughout the paper. Section~\ref{S:lemmas} establishes auxiliary results, 
including the density result, and presents examples.
Section~\ref{S:proofs} proves Theorems~\ref{ct:suff}~and~\ref{ct:ns} and their consequences. 
Finally, Section~\ref{S:interval}
focuses on the interval case, where we prove Theorem~\ref{ct:interval}  
and develop a more general version of that result.
\medskip

\section{Potential-theoretic preliminaries}\label{S:prelim}

In this section, we review the necessary background from logarithmic potential theory.
While the compact sets considered in Section~\ref{S:intro} are assumed to be symmetric with respect to the real axis, 
the potential-theoretic notions and results discussed here apply to arbitrary compact sets.
For a more comprehensive treatment, we refer the reader to \cite{ransford:ptCp95, SaffTotik:lpwef97}.
\smallskip

Throughout the paper, domains of integration are omitted for simplicity; integrals are understood over the 
support of the underlying measure, unless otherwise stated.
Let $\mu$ be a finite Borel measure on $\mathbb{C}$. The {\em logarithmic potential} of $\mu$ is defined by
\[
U^\mu(z):=\int \log \frac{1}{|z-w|}\, d\mu(w),
\]
and its {\em logarithmic energy} is given by
\[
I(\mu):=\int \int \log \frac{1}{|z-w|}\, d\mu(z)\, d\mu(w).
\]
If $\mu$ is compactly supported, then $-U^{\mu}$ is a subharmonic function on $\C$ that is harmonic
outside the support of $\mu$, and $I(\mu) > -\infty$.
These notions lead to the concept of logarithmic capacity. For a compact set $K\subset \mathbb{C}$, 
the {\em logarithmic capacity} is defined by
\[
\capa (K):=\exp \Big (-\inf_{\mu \in \mathcal{P}(K)} I(\mu)\Big).
\]
From a geometric perspective, the logarithmic capacity of $K$ admits an equivalent 
characterization as its transfinite diameter, defined by
\[
d_\infty(K) :=\lim_{n\to\infty}\max_{z_1,\ldots,z_n\in K} 
\bigg ( \prod_{1\le i<j\le n} |z_i-z_j|\bigg )^{\frac{2}{n(n-1)}}.
\]
Whenever $\capa(K)>0$, there exists a unique measure $\nu_K \in \mathcal{P}(K)$,
called the {\em equilibrium measure} of $K$, such that
\[
I(\nu_K)=\inf_{\mu \in \mathcal{P}(K)} I(\mu).
\]
The support of $\nu_K$ is contained in the boundary of the unbounded component of $\C\setminus K$.
For a circle, the equilibrium measure is the normalized arc-length measure.
For an arbitrary set $E\subset \C$, the logarithmic capacity is extended by
\[
\capa(E) :=\sup \{\capa(K) : K\subset E,\; K \text{ compact} \}.
\]
A property is said to hold {\em quasi-everywhere} on a set $E$
if it holds at every point of $E$ except possibly on a set of logarithmic capacity zero.
Frostman's theorem states that
\[
U^{\nu_K}(z)=-\log c(K)\quad \text{quasi-everywhere on }  K,
\]
while
\[
U^{\nu_K}(z)\leq -\log c(K)\quad \text{for all } z\in\mathbb{C}.
\]

We now turn to harmonic measure.
Let $\Omega$ be a domain in the Riemann sphere $\sph=\C \cup \{\infty\}$ whose boundary
$\partial \Omega$ has positive logarithmic capacity, and let $a \in \Omega$. 
The {\em harmonic measure of $\Omega$ with respect to $a$}, denoted by $\harmad$, 
is the unique Borel probability measure on
$\partial \Omega$ such that for every continuous function $f$ on $\partial\Omega$,
\[
u(a)=\int f(z)\, d\harmad (z),
\]
where $u$ is the harmonic function on $\Omega$ 
whose boundary extension agrees with $f$ quasi-everywhere on $\partial\Omega$.
Consequently, the harmonic measure of a Borel set 
$E\subset \partial\Omega$ is the value at $a$ of the harmonic 
function on $\Omega$ whose boundary values agree quasi-everywhere with
\[
\chi_E(z)=
\begin{cases}
1, & z \in E,\\
0, & z \in \partial\Omega\setminus E.
\end{cases}
\]
Harmonic measure can be interpreted probabilistically as the distribution of the first exit point of 
a Brownian particle starting at $a$ and moving inside $\Omega$.
\smallskip

We next recall the definition of the Green function. Let $\Omega\subset\sph$ be a domain, 
and let $a\in\Omega$. The Green function of $\Omega$ with pole at $a$, denoted by 
$g_\Omega(z,a)$, is the unique function satisfying
the following properties:
\begin{enumerate}[label=(\roman*)]
\item\label{I:firstpr} $g_{\Omega}(z,a)$ is nonnegative and harmonic in $\Omega \setminus \{a\}$ and 
bounded for $z$ outside of a neighborhood of $a$;
\item\label{I:secondpr} If $a \in \C$, then 
$g_{\Omega}(z,a) + \log |z-a|$
is bounded in a neighborhood of $a$, while if $a=\infty$, then
$g_{\Omega}(z,\infty)-\log |z|$ is bounded in a neighborhood of $\infty$;
\item For quasi-every $\zeta \in \partial\Omega$, we have
\[\displaystyle \lim_{z\to \zeta, z\in\Omega} g_{\Omega}(z,a)=0.\]
\end{enumerate}
Such a function exists if and only if $\partial \Omega$ is of positive logarithmic capacity.
We extend $g_{\Omega}$ to the whole Riemann sphere by setting
\begin{equation}\label{E:extension}
g_\Omega(\zeta,a):=
\begin{cases}
\displaystyle \limsup_{\substack{z\to\zeta, z\in\Omega}} g_\Omega(z,a),
& \zeta\in\partial\Omega, \\[1.5 ex]
0, & \zeta\notin\overline \Omega.
\end{cases}
\end{equation}
With this extension, $g_{\Omega}$ is subharmonic on $\sph \setminus \{a\}$
and satisfies
\[
g_\Omega(z,a)\geq 0,\qquad z\in\mathbb{C}\setminus\{a\}.
\]
A boundary point $\zeta \in\partial\Omega$ is called {\em regular} (with respect to $g_\Omega(z,a)$) if
\[
\lim_{z\to \zeta, z\in\Omega} g_{\Omega}(z,a)=0.
\]
This notion does not depend on the choice of $a$, and we therefore simply refer to 
regular boundary points. 
By the definition of the Green function, the set of irregular boundary points has logarithmic capacity zero.
A domain $\Omega$ is called {\em regular} if every boundary point is regular.
\smallskip

There is a representation formula connecting the Green function and harmonic measure 
(see \cite[page~103]{SaffTotik:lpwef97}) 
that will be useful in what follows.
Let $\Omega \subset \sph$ be a domain such that $\partial \Omega$ is of positive capacity.
For $a \in \Omega \setminus \{\infty\}$ and $z \in \sph \setminus \{a, \infty\}$, we have
\begin{equation}\label{E:green_harm}
g_{\Omega}(z,a)
=\log \frac{1}{|z-a|} - \int \log \frac{1}{|z-\zeta|}\, d \harmad (\zeta)+g_{\Omega}(a,\infty).
\end{equation}
Observe that if $\Omega$ does not contain point at $\infty$ then by our extended definition
of $g_{\Omega}$, we have $g_{\Omega}(a,\infty)=0$.
\smallskip

Let $K \subset \C$ be compact with positive logarithmic capacity, and let  $a \notin K$. 
To work directly with $K$, we adapt our notation as follows.
We write
\[
\harma:=\harmad.
\]
where $\Omega$ denotes the connected component of $\sph \setminus K$ containing $a$.
Since harmonic measure is supported on the boundary of its underlying domain, it follows that
$\harma$ is supported on the boundary $\partial K$.
Thus, although $\harma$ is defined using the connected component $\Omega$, 
it is naturally regarded as a measure on $K$.
In the special case $a=\infty$, harmonic measure coincides with the equilibrium measure of $K$:
\[
\harminfty= \nu_K.
\]
For $a \neq \infty$, the corresponding harmonic measure can be related to an equilibrium measure 
by means of a Möbius transformation. Specifically, consider the conformal map
\[
f(z)=\frac{1}{z-a}.
\]
Since harmonic measure is conformally invariant, we obtain
\[
\harma= f^*\nu_{f(K)},
\]
where $f^*$ denotes the pullback of measures (equivalently, the inverse pushforward $(f^{-1})_*$).
In this setting, we also write 
\[
g_K(z, a):=g_{\Omega}(z, a).
\]
In view of the extension of $g_\Omega$ defined in \eqref{E:extension}, 
the notation $g_K(z, a)$ is well defined whenever $a \notin K$ and $z \in \sph \setminus \{a\}$.
In particular, $g_K$ inherits the nonnegativity of $g_\Omega$, a property that will be used later.
It follows from Frostman's theorem that
\begin{equation}\label{E:green}
g_K(z, \infty)= -U^{\nu_K}(z) - \log \capa(K), \quad z \in \C.
\end{equation}
For $a \neq \infty$, the representation formula \eqref{E:green_harm} yields
\begin{equation}\label{E:green_harmK}
g_{K}(z,a)
=\log \frac{1}{|z-a|} - U^{\harma}(z) +g_{K}(a,\infty), \quad z \in \C \setminus \{a\}.
\end{equation}
Finally, we say that $K$ is {\em regular} if every connected component of $\sph \setminus K$ is regular 
in the sense defined above.
\medskip

\section{Auxiliary results}\label{S:lemmas}

This section collects auxiliary results, examples, and remarks that support and complement the 
main results of the paper. We begin by establishing two identities for logarithmic integrals, 
one with respect to the equilibrium 
measure and the other with respect to harmonic measure. Since $\mu_t(\alpha, K)$, defined in 
\eqref{E:conv_measure}, is a convex combination of these measures, these identities allow us to handle 
logarithmic integrals with respect to $\mu_t(\alpha, K)$. We also prove a density result 
for arithmetic probability measures of the form \eqref{E:conv_measure}.
\smallskip

The first identity gives a formula for the logarithmic integral of a polynomial with respect to 
the equilibrium measure.

\begin{lemma}\label{L:jensen}
Let $K \subset \C$ be a compact set with $\capa(K)>0$, and let $Q \in \C[z]$ be a nonzero polynomial. 
Then
\[
\int \log |Q(z)| \, d {\equi}(z)= \log |\ell_Q| + \deg Q \log \capa(K) + \sum_{Q(r)=0} g_K(r, \infty).
\]
where $\ell_Q$ is the leading coefficient of $Q$, and the sum is taken over the zeros of $Q$, 
counted with multiplicity.
\end{lemma}

\begin{proof}
We write $Q$ in its factored form
$$Q(z)=\ell_Q\prod_{j=1}^{\deg Q}(z-r_j),$$ where 
$r_1,\dots,r_{\deg Q}$ are the zeros of $Q$, counted with multiplicity. 
Taking logarithms of absolute values and integrating with respect to the equilibrium measure $\nu_K$,
we obtain
\[
\int \log |Q(z)| \, d {\equi}(z) 
= \log|\ell_Q| + \sum_{j=1}^{\deg Q} \int  \log |z-r_j| \, d {\equi}(z).
\]
By \eqref{E:green}, for every $r_j$, we have
\[
\int  \log |z-r_j| \, d {\equi}(z)=-U^{\nu_K}(r_j)
=\log\capa(K)+g_K(r_j,\infty).
\]
Substituting this into the preceding expression yields the desired formula.
\end{proof}

The following result establishes the corresponding identity for harmonic measure.
Unlike Lemma~\ref{L:jensen}, 
the resulting expression depends on the value of the polynomial at the pole of the harmonic measure 
and on the Green function relative to that pole.

\begin{lemma}\label{L:harm}
Let $K \subset \C$ be a compact set with $\capa(K)>0$, let $a \notin K$ and
let $Q \in \C[z]$ be a nonzero polynomial satisfying $Q(a)\neq 0$. 
Then
\[
\int \log |Q(z)| \, d \harma (z) = \log |Q(a)|+ \sum_{Q(r)=0} g_K(r, a)- \deg Q \cdot g_K (a, \infty), 
\]
where the sum is taken over the zeros of $Q$, 
counted with multiplicity.
\end{lemma}

\begin{proof}
Again, we write $Q$ in its factored form
\[
Q(z)=\ell_Q\prod_{j=1}^{\deg Q}(z-r_j).
\] 
Since $Q(a)\neq 0$, we have $r_j \neq a$ for any $j$. Using \eqref{E:green_harmK}, we obtain
\begin{align}
\int \log |Q(z)| \, d \harma (z) 
&= \log |\ell_Q| + \sum_{j=1}^{\deg Q}  \int \log |z-r_j| \, d \harma (z) \notag \\
&= \log |\ell_Q| + \sum_{j=1}^{\deg Q} \big(\log |a -r_j| + g_K(r_j, a)\big)- \deg Q\cdot g_K (a, \infty). \notag
\end{align}
Combining the logarithmic terms gives the desired identity.
\end{proof}

We next establish that the measures of the form \eqref{E:conv_measure} satisfying all the 
inequalities in \eqref{E:infinite} are dense, in the weak*
topology, among compactly supported Borel probability measures satisfying the same inequalities.

\begin{proposition}\label{P:dense}
Let $K \subset \C$ be compact and symmetric about the real axis. If $\mu\in\mathcal P(K)$
satisfies the inequalities in \eqref{E:infinite}, then there exists a sequence 
$\{\mu_n\}$ of probability measures of the form \eqref{E:conv_measure}, 
each satisfying the inequalities in \eqref{E:infinite}, such that $\mu_n\to\mu$ in the weak* topology.
\end{proposition}

\begin{proof}
Since $\mu \in \mathcal P(K)$ satisfy all the inequalities in \eqref{E:infinite}, 
Result~\ref{R:APM_char} implies that there exists a sequence of algebraic integers $\{\alpha_n\}$ 
satisfying conditions~\ref{I:deg}--\ref{I:weak} of Definition~\ref{D:APM}.
Passing to a subsequence if necessary, we may assume that the degrees are strictly increasing.
Let $\delta_{\conj(\alpha_n)}$ denote the normalized counting measure on $\conj(\alpha_n)$,
and let $P_n:= P_{\alpha_n}$.
Define
\[
f_n(z):= \frac {{P_n(z)}^{(\deg P_{n+1})+1}}{{P_{n+1}(z)}^{\deg P_n}}.
\]
We then define the measures
\[
\mu_n:= f_n^* (m_{S^1}),
\]
where $m_{S^1}$ is the normalized arc-length measure on the unit circle, and 
$f_n^*$ denotes the normalized pull-back.
It is shown in \cite{BMQS:cgemh26} that each $\mu_n$ satisfies all the inequalities in \eqref{E:infinite}
and that $\mu_n \to \mu$ in the weak* topology as $n \to \infty$.
Indeed, we have
\[
U^{\mu_n}= -\max \left \{\frac{\log |P_n|}{\deg P_n}, \frac{\log |P_{n+1}|}{\deg P_{n+1}} \right\}
= -\min \left \{U^{\delta_{\conj(\alpha_n)}}, U^{\delta_{\conj(\alpha_{n+1})}} \right\}.
\]
Since $U^{\delta_{\conj(\alpha_n)}}$ converges to $U^{\mu}$ in $L^1_{loc}(\C)$, it follows that 
$U^{\mu_n}$ also converges to $U^{\mu}$ in $L^1_{loc}(\C)$. Thus
$\mu_n \to \mu$ in the weak* topology.
Since the degrees of $\alpha_n$ are increasing, the sets 
\[
K_n:= \operatorname{supp} (\mu_n) = f_n^{-1} (S^1)
\]
are compact.
Since each $P_n \in \Z[z]$, each $K_n$ is symmetric with respect to the real axis.
\smallskip

Finally, by the definition of $\mu_n$, we see that it is a convex combination of the harmonic 
measures on $K_n$ with respect to the poles of $f_n$, where the weights are proportional 
to the orders of these poles. The poles of $f_n$ are exactly the Galois conjugates of 
$\alpha_{n+1}$ together with $\infty$. Furthermore, all finite poles of $f_n$ have the same order. 
Since the harmonic measure of $K_n$ with respect to $\infty$ coincides with the equilibrium 
measure of $K_n$, it follows that $\mu_n$ has the form given in \eqref{E:conv_measure}, 
completing the proof.
\end{proof}

\begin{remark}\label{rem:negativepotential}
If $\capa(K) =1$, then the equilibrium measure $\equi$ is the 
only arithmetic probability measure on $K$. To see this, let $\{\alpha_n\}$ be a sequence of algebraic 
integers satisfying conditions \ref{I:deg} and \ref{I:neighbour} in Definition~\ref{D:APM}. Let
$d_n= \deg \alpha_n$, and denote by $\gamma_{n,1}, \dots, \gamma_{n, d_n}$ the Galois conjugates
of $\alpha_n$. Since
\[
\operatorname{Disc}(\alpha_n)=\prod_{i<j}(\gamma_{n,i}-\gamma_{n,j})^2
\]
is a nonzero integer, any weak$^*$ subsequential limit $\mu$ of
the normalized counting measure on $\conj(\alpha_n)$ satisfies
\[
I(\mu)\leq \liminf_{n\to\infty}
-\frac{1}{d_n(d_n-1)}\log|\operatorname{Disc}(\alpha_n)|\leq 0.
\]
Since $\equi$ is the unique measure in $\mathcal P(K)$ satisfying $I(\equi) \leq 0$,
it follows that $\equi$ is the unique arithmetic probability measure in this case.
In fact, such convergence holds under the weaker assumption that the height
of $\alpha_n$ relative to $K$ tends to zero; see \cite{Bi, Ru, Pritsker:crelle11}.
\smallskip

When $\operatorname{cap}(K)>1$,
one may naturally try to construct arithmetic probability measures by taking the weak$^*$ closure of the 
set of equilibrium measures of compact subsets of $K$ that are symmetric with respect to the real axis 
and have logarithmic capacity equal to one.
This construction, however, produces only a restricted class of arithmetic probability 
measures, namely, symmetric measures $\mu\in\mathcal P(K)$ satisfying
\[
U^\mu(z)\leq 0 \qquad \text{for every } z\in\mathbb{C}.
\]
Since there are arithmetic probability measures whose potentials are positive at some points, 
the class of compactly supported symmetric probability measures with nonpositive potentials 
is not dense in the space of all compactly supported arithmetic probability measures.
We refer the reader to \cite{NT:cdmcav25, BMQS:cgemh26} for further details.
\end{remark}

\begin{remark}\label{rem:smallt}
The second condition in Theorem~\ref{ct:suff} reflects the position of the Galois conjugates of 
$\alpha$ relative to $K$. To make this dependence explicit, define
\[
h_K(\alpha):= \frac{1}{\deg \alpha} \Big (\log |\ell_{\alpha}|
+\sum_{\gamma \in \conj(\alpha)} g_K(\gamma, \infty) \Big).
\]
By Lemma~\ref{L:jensen}, with $Q= P_{\alpha}$, the second condition in Theorem~\ref{ct:suff} becomes
\[
(1-t)\Big (\log \capa(K)+ h_K(\alpha) - \frac{1}{\deg \alpha}\log^+ \minalp \Big ) \leq \log \capa(K).
\]
Thus, the size of $h_K(\alpha)$ directly influences the restriction on $t$.
Its geometric dependence on the location of the conjugates is particularly transparent when $K$ is regular. 
In this case, $g_K(\gamma, \infty) \to 0$ as $\gamma$ approaches $K$ from the unbounded 
component, while $g_K(\gamma, \infty) \to \infty$ as $|\gamma| \to \infty$.
Hence, for algebraic integers, $h_K(\alpha)$ becomes small when all the Galois conjugates lie close to
$K$, whereas conjugates lying farther from $K$ contribute more substantially to
$h_K(\alpha)$.
\smallskip

Now suppose that $\capa(K)>1$, without assuming regularity, and $t_{\min}$ denote the threshold 
imposed by the second condition. Since $\log^+ \minalp \geq 0$,
\[
t_{\min} \leq \frac{h_K(\alpha)}{\log \capa(K)+h_K(\alpha)}.
\]
Consequently, $t_{\min} \to 0$ as $h_K(\alpha) \to 0$.
Hence, as $h_K(\alpha)$ decreases to 0, the second condition becomes progressively 
less restrictive, and the first condition becomes the more significant constraint.
Moreover, the Fekete--Szeg\H{o} results recalled in Section~\ref{S:intro} 
show that this limiting regime is non-vacuous:
if $\capa(K) \geq 1$, there exists a sequence of algebraic integers $\{\alpha_n\}$
with
$$h_K(\alpha_n) \to 0.$$
The quantity $h_K(\alpha)$ has been studied in 
various works; see, for instance, \cite{Ru, Pritsker:crelle11, LevLon:evoft24}.
\end{remark}

We now show that, in Theorem~\ref{ct:ns}, neither of the two conditions
\[\int \log |P_{\alpha}(z)| \, d \mu_t(z) \geq 0 \quad {\text{and}} \quad
\int \log |P_{\beta}(z)| \, d \mu_t(z) \geq 0\]
can be omitted. In other words, the two conditions are logically independent: each can hold while the other fails. 
We illustrate this in the case $\alpha=0$ and $\beta=1$.
In this setting, $P_0(z)=z$ and $P_1(z)=z-1$, and
$$\mu_t= t \equi + (1-t) \omega(0, K).$$
\begin{example}\label{Ex:both_condi}
We first show that the inequality corresponding to $P_0(z)=z$ alone is not sufficient
to ensure that $\mu_t$ satisfies all the inequalities in \eqref{E:infinite}.
Let $K=[1,5]$. Since $\capa (K)=1$, as observed  in
Remark~\ref{rem:negativepotential}, the equilibrium measure $\nu_K$
is the unique measure on $K$ satisfying all the inequalities in \eqref{E:infinite}.
Since $\log|z|>0$ on $(1,5]$, we have
\[
\int \log |z|\,d\nu_K(z)>0.
\]
By continuity of the family $\mu_t$ and the fact that $\mu_1=\nu_K$, there
exists $0\leq t<1$ such that
\[
\int \log |z|\,d\mu_t(z)\geq0.
\]
However, $\mu_t$ cannot satisfy all the inequalities in \eqref{E:infinite} for any
$t<1$, since $\nu_K$ is the unique measure on $K$ with this property. 
Thus, the inequality corresponding to $P_0$ can hold for some $t <1$ without implying that 
$\mu_t$ satisfies all the inequalities in \eqref{E:infinite}.
\smallskip

We next show that the inequality corresponding to $P_1(z)=z-1$ alone is likewise not sufficient.
Let $K=[\epsilon,\epsilon+8]$, where $\epsilon>0$ will be chosen later.
By Lemmas~\ref{L:jensen}~and~\ref{L:harm},
\[
\int \log |z-1| \, d \mu_t(z)= t \log \capa (K)- (1-t) g_K(0, \infty)
=t \log 2- (1-t) g_K(0, \infty).
\]
As $\epsilon \to 0$, we have $g_K(0, \infty) \to 0$. 
Thus the values of $t$ for which this integral is nonnegative can be
chosen arbitrarily close to zero.
On the other hand, using a limiting version of Lemma~\ref{L:harm} together with the
inversion relation for the Green function, we obtain
\[
\int \log |z|\, d\harmz(z)= \lim_{y \to 0 } (g_K(y, 0)+ \log |y|)- g_K(0, \infty)= - \log \capa (1/K)- g_K(0, \infty),
\]
where $1/K= [1/(8+\epsilon), 1/\epsilon]$ denotes the inverted set. Therefore,
\[
\int \log |z| \, d \mu_t(z)=t \big (\log 2 + g_K(0, \infty)\big)- (1-t) \big (\log \capa (1/K)+ g_K(0, \infty)\big).
\]
Since $\log \capa (1/K) \to \infty$ and $g_K(0, \infty) \to 0$ as $\epsilon \to 0$,
we can choose $\epsilon > 0$ sufficiently small and a
value of $t$ for which
$$\int \log |z-1| \, d \mu_t(z) \geq 0 \text{ \ while } \int \log |z|\, d\mu_t(z)<0.$$
Thus, the inequality corresponding to $P_1$ can hold even though the inequality corresponding to $P_0$ fails.
Thus, the inequality corresponding to $P_1$ alone does not suffice.
\smallskip

Together, these two constructions show that neither of the two logarithmic inequalities in Theorem~\ref{ct:ns}
can be omitted.
\hfill $\blacktriangleleft$
\end{example}

We conclude this section with the example of a circle, for which the potential-theoretic quantities 
appearing above admit explicit formulas. This provides a concrete illustration of the arithmetic probability 
measures.
 
\begin{example}\label{Ex:notnecce}
Let 
\[
K=\{z\in\mathbb{C}:|z|=r\}.
\]
Its logarithmic capacity is $\capa(K)=r$, and its Green function with pole at infinity is
\[
g_K(z,\infty)=\log^+ \frac{|z|}{r}.
\]
By rotational symmetry, the equilibrium measure is the normalized arclength measure on $K$.
The harmonic measure is also explicit. If $|a|\neq r$, then
\[
d\harma(re^{i\theta})
=
\frac{1}{2\pi}
\frac{\left|r^2-|a|^2\right|}
{\left|re^{i\theta}-a\right|^2}
\,d\theta.
\]
Thus, for a measurable set $E\subset K$, we have
\[
\harma(E)
=
\frac{1}{2\pi}
\int_{\{\theta:\,re^{i\theta}\in E\}}
\frac{\left|r^2-|a|^2\right|}
{\left|re^{i\theta}-a\right|^2}
\,d\theta.
\]
Taking the limit as $|a| \to \infty$, we obtain the harmonic measure with respect to $\infty$:
\[
d\harminfty (re^{i\theta})
=\frac{d\theta}{2\pi},
\]
which agrees with the normalized arclength measure on $K$.
\smallskip

We next give an explicit expression for the integral involving $P_{\alpha}$ in this setting.
Let $\gamma_1, \dots, \gamma_{\deg \alpha}$ denote the Galois conjugates of $\alpha$,
and define
\[
\tau_i=
\begin{cases}
\gamma_i/r, & |\gamma_i|< r,\\[5pt]
r/\overline{\gamma_i}, &  |\gamma_i|> r.
\end{cases}
\]
Then
\[
\int \log|P_\alpha(z)|\, d\mu_t(z)= \log |\ell_{\alpha}|+ \sum_{i}\log \max\{r, |\gamma_i|\}
+\frac{1-t}{\deg \alpha}\sum_{i} \sum_{j} |1-\tau_i \overline{\tau_j}|.
\]
These explicit formulas, together with our results, provide a concrete 
family of arithmetic probability measures on a circle.
\hfill $\blacktriangleleft$
\end{example}

\medskip

\section{Proofs of Theorems~\ref{ct:suff} and \ref{ct:ns}}\label{S:proofs}

In this section, we prove Theorems~\ref{ct:suff} and \ref{ct:ns}, together with their corollaries. 
We begin by recalling the standing setting. Let 
$K\subset\mathbb C$ be a compact set with $\capa(K)>0$, symmetric with respect to the real axis, 
and let $\alpha$ be an algebraic number satisfying
\[
\conj(\alpha)\cap K=\emptyset.
\]
For each $t\in[0,1]$, let $\mu_t(\alpha,K)$ be the measure defined in \eqref{E:conv_measure}.
\smallskip

We first establish a lower bound
for logarithmic integrals with respect to $\mu_t(\alpha,K)$.
The proof of the following theorem relies on the decomposition of $\mu_t(\alpha,K)$ 
into its equilibrium and harmonic components, together with the nonnegativity of the 
Green function and the integrality of the resultant.

\begin{lemma}\label{L:lower_arith}
For $t \in [0,1]$, let $\mu_t:= \mu_t(\alpha, K)$ denote the probability measure defined in \eqref{E:conv_measure}.
If $Q \in \Z[z]$ satisfies $Q(\alpha)\neq 0$, then
\[
\frac{1}{\deg Q}\int \log |Q(z)| \, d \mu_t(z)
\geq t \log \capa(K)- \frac{1-t}{\deg \alpha} \Big(\log |\ell_{\alpha}| + \sum_{\gamma \in \conj(\alpha)}  
g_K(\gamma, \infty) \Big),
\]
where $\ell_{\alpha}$ denotes
the leading coefficient of the polynomial $P_{\alpha}$.
\end{lemma}

\begin{proof}
Let $Q\in \Z[z]$ satisfy $Q(\alpha)\neq0$. We estimate the two components of $\mu_t$ separately.
By definition, the Green function $g_K(r, \infty)$ is nonnegative for every $r \in\C$.
Since $Q\in\Z[z]$, its leading coefficient satisfies $|\ell_Q|\ge1$. 
Thus, Lemma~\ref{L:jensen} gives
\begin{equation}\label{E:ineq_equi}
\int \log |Q(z)| \, d {\equi}(z) \geq  \deg Q \log \capa(K).
\end{equation}
Let $\gamma\in\conj (\alpha)$. By the standing assumption, $\gamma\notin K$.
Thus, by definition of the Green function, $g_K(r,\gamma)\ge0$ for every $r\in\mathbb C$.
Furthermore, since $Q\in\Z[z]$ and $Q(\alpha)\neq0$, we have $Q(\gamma)\neq0$.
Therefore, Lemma~\ref{L:harm} gives
\begin{equation}\label{E:ineq_unbou}
\int \log |Q(z)| \, d {\harm}(z) \geq \log |Q(\gamma)|- \deg Q \cdot g_K (\gamma, \infty). 
\end{equation}
Summing \eqref{E:ineq_unbou} over all $\gamma\in\mathcal O(\alpha)$ and using
\[
|\res (Q, P_{\alpha})|=|\ell_{\alpha}|^{\deg Q}\prod_{\gamma \in \conj(\alpha)} |Q(\gamma)|,
\]
we obtain
\[
\sum_{\gamma \in \conj(\alpha)} \int \log |Q(z)| \, d {\harm}(z) 
\geq
\log \frac{|\res (Q, P_{\alpha})|}{|\ell_{\alpha}|^{\deg Q}} 
- \deg Q \sum_{\gamma \in \conj(\alpha)} g_K(\gamma, \infty).
\]
Since $Q$ and $P_{\alpha}$
have integer coefficients, $\res (Q, P_{\alpha})$ is an integer.
Because $Q(\gamma) \neq 0$ for every $\gamma \in \conj(\alpha)$, 
the resultant is nonzero, and hence
\[
|\res (Q, P_{\alpha})| \geq 1.
\]
Thus the logarithmic term involving the resultant is nonnegative and 
may be discarded. 
Combining the resulting inequality with \eqref{E:ineq_equi} and the definition of $\mu_t$ gives
\[
\int \log |Q(z)| \, d \mu_t(z)
\geq t  \deg Q \log \capa(K)
-\frac{1-t}{\deg \alpha} \Big(\deg Q \log |\ell_{\alpha}| 
+ \deg Q \sum_{\gamma \in \conj(\alpha)} g_K(\gamma, \infty) \Big).
\]
Dividing by $\deg Q$ proves the lemma.
\end{proof}

We first prove
Theorem~\ref{ct:ns}, which will follow as an immediate consequence of the more general result below.  
Recall from Section~\ref{S:prelim} that $K$ is regular if every connected component of
$\C \setminus K$ is regular. 
In particular, the regularity condition appearing below will be automatic 
under the hypothesis of Theorem~\ref{ct:ns}.

\begin{theorem}\label{T:ns}
Suppose $K$ is regular and there exists an algebraic integer $\beta$ such that
$P_{\alpha}(\beta)$ is an algebraic unit, and each element of $\conj(\beta)$ either lies in $K$ or in a bounded 
connected component of $\C\setminus K$ containing no conjugate of $\alpha$.
Then $\mu_t:=\mu_t(\alpha,K)$ satisfies all the inequalities in \eqref{E:infinite}
if and only if
\[\int \log |P_{\alpha}(z)| \, d \mu_t(z) \geq 0 \quad {\text{and}} \quad
\int \log |P_{\beta}(z)| \, d \mu_t(z) \geq 0.\]
\end{theorem}

\begin{proof}
Since $P_{\alpha}, P_{\beta} \in \Z[z]$, the forward implication is immediate. 
For the converse, we first show that, for every nonzero
$Q \in \Z[z]$ with $Q(\alpha)\neq 0$, 
\begin{equation}\label{E:ineq_beta}
\frac{1}{\deg Q} \int \log |Q(z)| \, d \mu_t(z) \geq \frac{1}{\deg P_{\beta}} \int \log |P_{\beta}(z)| \, d \mu_t(z).
\end{equation}
We show that the lower bound in Lemma~\ref{L:lower_arith} is attained by $P_{\beta}$.
Since $\beta$ is an algebraic integer, $P_{\beta}$ is monic. 
Moreover, since $K$ is regular, $g_K(z,\infty)=0$ for every $z$ outside the unbounded component of 
$\C\setminus K$. By hypothesis, every element of $\conj(\beta)$ lies outside this unbounded component.
Hence
\[
g_K(\eta, \infty)=0 \quad \text {for every } \eta \in \conj(\beta).
\]
Hence Lemma~\ref{L:jensen} gives
\[
\int \log |P_{\beta}(z)| \, d {\equi}(z) = \deg P_{\beta} \log \capa(K).
\]
Now let $\gamma \in \conj(\alpha)$.
By hypothesis, no element of $\conj(\beta)$ lies in the same component of $\C \setminus K$ as $\gamma$.
Thus the regularity of $K$ implies that
\[
g_K(\eta, \gamma)=0 \quad \text {for every } \eta \in \conj(\beta).
\]
Since $P_{\beta}(\gamma) \neq 0$, Lemma~\ref{L:harm} gives
\[
\int \log |P_{\beta}(z)| \, d {\harm}(z) = \log |P_{\beta}(\gamma)|- \deg P_{\beta} \cdot g_{K} (\gamma, \infty). 
\]
Since $P_\beta$ is monic,
$\res (P_\beta,P_\alpha) =N_{\mathbb Q(\beta)/\mathbb Q}\bigl(P_\alpha(\beta)\bigr)$.
By assumption, $P_\alpha(\beta)$ is an algebraic unit, so its norm is $\pm1$. Therefore,
\[
|\res (P_{\beta}, P_{\alpha})|=1.
\]
Combining these equalities with the definition of $\mu_t$, we obtain
\[
\frac{1}{\deg P_{\beta}}\int \log |P_{\beta}(z)| \, d \mu_t(z)
= t \log \capa(K)- \frac{1-t}{\deg \alpha} \Big(\log |\ell_{\alpha}| + 
\sum_{\gamma \in \conj(\alpha)} g_K(\gamma, \infty)\Big).
\]
This is precisely the lower bound in Lemma~\ref{L:lower_arith}. Hence \eqref{E:ineq_beta} follows.
By hypothesis, the right hand side of \eqref{E:ineq_beta} is nonnegative. Thus
 \[
 \int \log |Q(z)| \, d \mu_t(z) \geq 0
 \]
 for every $Q\in \Z[z]$ satisfying $Q(\alpha)\neq 0$.
\smallskip

It remains to consider a nonzero $Q \in \Z[z]$ with $Q(\alpha) = 0$. 
Since $P_{\alpha}$ is irreducible and $P_{\alpha}(\alpha)=0$, we can write
$$
Q(z)= P_{\alpha}(z)^k R(z)
$$
for some $k \geq 1$ and $R \in \Z[z]$ 
satisfying $R(\alpha) \neq 0$. Therefore,
\[
\int \log |Q(z)| \, d \mu_t(z) = k \int \log |P_{\alpha}(z)| \, d \mu_t(z) + \int \log |R(z)| \, d \mu_t(z). 
\]
The first term is nonnegative by hypothesis, while the second is nonnegative by the preceding argument,
since $R(\alpha) \neq 0$. Hence
$$
\int \log |Q(z)| \, d \mu_t(z)\geq 0.
$$
Therefore, $\mu_t$ satisfies all the inequalities in \eqref{E:infinite}.
\end{proof}

We now deduce Theorem~\ref{ct:ns} from Theorem~\ref{T:ns}.

\begin{proof}[Proof of Theorem~\ref{ct:ns}]
Since every connected
component of $K$ contains more than one point, by \cite[Theorem 3.8.3]{ransford:ptCp95},
each connected component $C$ is non-thin at every point of $\overline C=C$.
Hence $K$ is non-thin at every point of $K$. 
By \cite[Theorem 4.2.4]{ransford:ptCp95}, $K$ is regular, and hence Theorem~\ref{T:ns} applies.
\end{proof}

\begin{remark}
A careful reader will observe that the full strength of the regularity assumption on $K$ is not needed 
in the proof of Theorem~\ref{T:ns}. It suffices to require that every ellement of $\conj(\beta)$ lying in $\partial K$ 
be a regular point. We assume $K$ regular only to simplify the statement.
\end{remark}

We next consider the particularly important case
$\alpha=0$. In this case,
\[
\mu_t(0,K)=t\nu_K+(1-t)\harmz,
\]
and such measures arise naturally in connection with the Schur--Siegel--Smyth trace problem;
see \cite{Smith:aicpd24, LL:SSS26}. 
Theorem~\ref{T:ns} then gives the following criterion.

\begin{corollary}\label{C:units}
Suppose K is regular and there exists an algebraic unit $\beta$ such that 
each element of $\conj(\beta)$ either lies in $K$ or
in a bounded component of $\C \setminus K$ not containing 0.
Then $\mu_t:= \mu_t (0, K)$
is an arithmetic probability measure if and only if
\[
\int \log |z|\,d\mu_t(z)\geq 0
\quad\text{and}\quad
\int \log |P_{\beta}(z)| \,d\mu_t(z)\geq 0.
\]
\end{corollary}

We next derive from Lemma~\ref{L:lower_arith} a criterion that does not involve the auxiliary 
algebraic integer $\beta$. Under the following additional hypotheses, the lower bound in
Lemma~\ref{L:lower_arith} simplifies substantially. 
We restrict to the case $\capa(K)\geq 1$, since when $\capa(K)<1$, 
no probability measure on $K$ satisfies all the inequalities in \eqref{E:infinite}.

\begin{corollary}
Let $\capa(K) \geq 1$, and let $\alpha$ be an algebraic integer such that no element of $\conj(\alpha)$ 
lies in the unbounded component of $\mathbb{C}\setminus K$. 
Then $\mu_t:=\mu_t(\alpha,K)$ satisfies all the inequalities in \eqref{E:infinite}
if and only if
\[
\int \log |P_{\alpha}(z)|\,d\mu_t(z)\geq 0.
\]
\end{corollary}

\begin{proof}
The necessity is immediate. For the converse, let $Q \in \Z[z]$ satisfy $Q(\alpha) \neq 0$.
By Lemma~\ref{L:lower_arith}, we have
\[
\frac{1}{\deg Q}\int \log |Q(z)| \, d \mu_t(z)
\geq t \log \capa(K)- \frac{1-t}{\deg \alpha} \Big(\log |\ell_{\alpha}| + \sum_{\gamma \in \conj(\alpha)}  
g_K(\gamma, \infty) \Big).
\] 
Since $\alpha$ is an algebraic integer, $\ell_{\alpha}=1$.
Moreover, the standing assumption $\conj(\alpha)\cap K=\emptyset$, 
together with the hypothesis of the corollary, implies that every conjugate of $\alpha$ lies in a 
bounded component of $\C\setminus K$. Hence
$g_K(\gamma, \infty)=0$ for every $\gamma \in \conj(\alpha)$. Therefore,
\[
\frac{1}{\deg Q}\int \log |Q(z)| \, d \mu_t(z)\geq t \log \capa(K) \geq 0,
\]
since $\capa(K) \geq 1$.
If $Q(\alpha)=0$, factor $Q=P_\alpha^kR$ as in the proof of Theorem~\ref{T:ns}, 
with $k\ge1$ and $R \in \Z[z]$ satisfying $R(\alpha)\neq0$. The required inequality then follows from the hypothesis 
for $P_\alpha$ and the preceding argument for $R$.
\end{proof}

We now turn to the proof of Theorem~\ref{ct:suff}.
The main step is to obtain a lower bound for the logarithmic integral of an arbitrary polynomial
$Q\in\mathbb Z[z]$ in terms of the minimum
\[
\minalp:= \min_{z \in K} |P_{\alpha}(z)|=\min_{z \in \partial K} |P_{\alpha}(z)|.
\]
introduced in Section~\ref{S:intro}. This quantity captures the position of the Galois conjugates of 
$\alpha$ relative to $\partial K$.
We first recall the following version of the generalized minimum principle.
\begin{result}[\cite{ransford:ptCp95}, Theorem 3.6.9]
Let $\Omega\subset\C$ be a domain with $\capa(\partial \Omega)>0$,
and let $u$ be superharmonic in $\Omega$ and bounded below. If
\[\liminf_{z\to\zeta}u(z)\geq 0\] for quasi every $\zeta\in\partial\Omega$, then \[u(z)\geq 0,\qquad z\in\Omega.\]
\end{result}

We now establish the required lower bound.

\begin{lemma}\label{L:lower_geo}
For $t \in [0,1]$, let $\mu_t:= \mu_t(\alpha, K)$ denote the probability measure defined in \eqref{E:conv_measure}.
If $Q \in \Z[z]$ satisfies $Q(\alpha)\neq 0$, then
\[
\frac{1}{\deg Q}\int \log |Q(z)| \, d \mut(z)
\geq  t \log \capa(K) - \frac{1-t}{\deg \alpha} \bigg(\log \ell_{\alpha} - \log \minalp
+ \sum_{\gamma \in \conj(\alpha)} g_K(\gamma, \infty) \bigg),
\]
where $\ell_{\alpha}$ denotes
the leading coefficient of the polynomial $P_{\alpha}$.
\end{lemma}

\begin{proof}
Let $Q$ be a polynomial in $\Z[z]$
satisfying $Q(\alpha)\neq 0$. 
Since $Q$ has integer coefficients, the condition
$Q(\alpha) \neq 0$ implies that $Q(\gamma) \neq 0$ for every $\gamma \in \conj(\alpha)$.
We write $Q$ in its factored form
$$
Q(z)=\ell_Q\prod_{j=1}^{\deg Q}(z-r_j).
$$
By Lemmas~\ref{L:jensen}~and~\ref{L:harm}, we get
\begin{equation}\label{E:jen_harm}
\begin{aligned}
\int \log |Q(z)| \, d \mut(z)
= &t \Big(\log |\ell_Q| + \deg Q \log \capa(K) + \sum_{j=1}^{\deg Q} g_K(r_j, \infty) \Big) \\
&+\frac{1-t}{\deg \alpha}\sum_{\gamma \in \conj(\alpha)}\Big(  \log |Q(\gamma)|- \deg Q g_K(\gamma, \infty) 
+ \sum_{j=1}^{\deg Q} g_K(r_j, \gamma) \Big).
\end{aligned}
\end{equation}
The dependence on $Q$ enters only through its leading coefficient and its zeros. 
To estimate the root dependent contribution, define
\[
H_t(z):= t g_K(z, \infty)+ \frac{1-t}{\deg \alpha} \sum_{\gamma \in \conj(\alpha)} 
\big(\log |z - \gamma|+ g_K(z, \gamma)\big).
\] 
Note that $H_t$ is well defined on $\C \setminus \conj(\alpha)$.
With this notation, \eqref{E:jen_harm} becomes
\begin{equation}\label{E:intQ1}
\begin{aligned}
\frac{1}{\deg Q}\int \log |Q(z)| \, d \mut(z)
= t \log \capa(K) - \frac{1-t}{\deg \alpha} & \sum_{\gamma \in \conj(\alpha)} g_K(\gamma, \infty) \\
&+ \frac{\log|\ell_Q|}{\deg Q} +\frac{1}{\deg Q} \sum_{j=1}^{\deg Q}  H_t(r_j).
\end{aligned}
\end{equation}

We first establish that $H_t$ is harmonic on $\C \setminus \partial K$. 
Fix $\gamma\in \conj(\alpha)$, and let $\Omega_{\gamma}$ be the connected component 
of $\C \setminus \partial K$ containing $\gamma$. 
By condition \ref{I:firstpr}  in the definition of the Green function, $g_K(z, \gamma)$ is harmonic 
on $\Omega_\gamma\setminus{\gamma}$. By condition \ref{I:secondpr}, the function
\[
g_K(z, \gamma) + \log\vert{}z - \gamma\vert{}
\]
is bounded in a neighborhood of $\gamma$. 
Hence the singularity at $z = \gamma$ is removable, and this function extends 
harmonically to all of $\Omega_{\gamma}$.
If $z$ belongs to any other component of $\mathbb{C} \setminus \partial K$, then
$g_K(z, \gamma) = 0$.
Since $z\neq\gamma$ there, $\log|z-\gamma|$ is harmonic as well. We conclude that
\[
g_K(z, \gamma) + \log\vert{}z - \gamma\vert{}
\]
is harmonic throughout $\mathbb{C} \setminus \partial K$.
Since $g_K(z, \infty)$ is also harmonic on $\mathbb{C} \setminus \partial K$,
it follows that $H_t$ is harmonic on $\mathbb{C} \setminus \partial K$. 
\smallskip

We next show that the infimum of $H_t$ over $\mathbb C$ is determined by its values on $\partial K$.
Since
$g_K(z, \infty) \to \infty$ and $\log \vert{}z - \gamma \vert{} \to \infty$ 
as $\vert{}z\vert{} \to \infty$, we have
\[
\lim_{\vert{}z\vert{} \to \infty} H_t(z) = +\infty
\]
independent of the parameter $t \in [0,1]$. 
Now consider an arbitrary connected component $\Omega$ of $\C\setminus\partial K$. \
The function $H_t$ is harmonic and bounded below in $\Omega$, with finite boundary limits 
quasi-everywhere on $\partial\Omega$. The generalized minimum principle therefore gives
\[
\inf_{\Omega}H_t\geq \inf_{\partial\Omega}H_t.
\]
Since $\partial\Omega\subseteq\partial K$, it follows that
\[
H_t(z)\geq\inf_{\zeta\in\partial K}H_t(\zeta),
\qquad z\in\mathbb C.
\]
By the definition of $H_t$ and the nonnegativity of the Green functions,
\[
\inf_{\zeta \in \partial K} H_t(\zeta)
\geq \inf_{\zeta \in \partial K} \frac{1-t}{\deg \alpha} \sum_{\gamma \in \conj(\alpha)} \log |\zeta -\gamma|
= \frac{1-t}{\deg \alpha} (\log \minalp -\log \ell_{\alpha}).
\]
Therefore, we get
\[
H_t(z)
\geq
\frac{1-t}{\deg\alpha}
\bigl(\log\minalp-\log\ell_\alpha\bigr),
\qquad z\in\mathbb C.
\]
Using this in \eqref{E:intQ1} together with $|\ell_Q| \geq 1$, gives us
\[
\frac{1}{\deg Q}\int \log |Q(z)| \, d \mut(z)
\geq  t \log \capa(K) - \frac{1-t}{\deg \alpha} \bigg(\log \ell_{\alpha} - \log \minalp
+ \sum_{\gamma \in \conj(\alpha)} g_K(\gamma, \infty) \bigg),
\]
which completes the proof.
\end{proof}

We now apply Lemmas~\ref{L:lower_arith}~and~\ref{L:lower_geo} to prove Theorem~\ref{ct:suff}.

\begin{proof}[Proof of Theorem~\ref{ct:suff}]
Let $Q\in\Z[z]$ be nonzero and satisfy $Q(\alpha)\neq0$.
Combining the lower bounds in 
Lemmas~\ref{L:lower_arith}~and~\ref{L:lower_geo}, we obtain
\[
\frac{1}{\deg Q}\int \log |Q(z)| \, d \mu_t(z)
\geq t \log \capa(K)- \frac{1-t}{\deg \alpha} \Big(\log |\ell_{\alpha}| + \sum_{\gamma \in \conj(\alpha)}  
g_K(\gamma, \infty) - \log^+\minalp \Big).
\]
The expression on the right hand side can be rewritten as
\[
\log \capa(K) - \frac{1-t}{\deg \alpha} \Big (\deg \alpha \log \capa(K) + \log |\ell_{\alpha}|
+\sum_{\gamma \in \conj(\alpha)} g_K(\gamma, \infty) - \log^+\minalp \Big).
\]
We invoke Lemma~\ref{L:jensen} with $Q= P_{\alpha}$. Since $\deg P_{\alpha}=\deg \alpha$,
it follows that
\[
\frac{1}{\deg Q}\int \log |Q(z)| \, d \mu_t(z) \geq \log \capa(K)
- \frac{1-t}{\deg \alpha}\Big( \int \log |P_{\alpha}(z)| \, d {\equi}(z) - \log^+\minalp \Big).
\]
By assumption, the right hand side is nonnegative. 
Hence
$$\int \log |Q(z)| \, d \mu_t(z)\geq 0$$ 
whenever $Q(\alpha) \neq 0$.
\smallskip

Finally, suppose $Q(\alpha)=0$. As in the proof of Theorem~\ref{T:ns}, write
\[
Q(z)= P_{\alpha}(z)^k R(z),
\]
where $k \geq 1$ and $R \in \Z[z]$ with $R(\alpha) \neq 0$.
The integral corresponding to $P_\alpha$ is nonnegative by assumption, 
while that corresponding to $R$ is nonnegative by the preceding argument. Thus
\[
\int \log |Q(z)| \, d \mu_t(z)\geq 0.
\]
Therefore, $\mu_t$ satisfies all the inequalities in \eqref{E:infinite} and
by Result~\ref{R:APM_char}, $\mu_t$ is an arithmetic probability measure.
\end{proof}

\begin{proof}[Proof of Corollary~\ref{cor:faraway}]
Let $\alpha$ be an algebraic number and suppose that every $\gamma \in \conj(\alpha)$ satisfies 
$|\gamma|\geq r$, where $r > 0$, depending only on $K$, will be chosen sufficiently large.
Since $\nu_K$ is a probability measure,
\[
\int \log |z-w| \, d\nu_K(z)=  \log|w| + \int \log \left \vert 1-\frac{w}{z} \right \vert \, d\nu_K(z).
\]
Because $\nu_K$ has compact support, the second integral is $o(1)$ as $w \to \infty$. 
Moreover, since $K$ is compact,
\[
\log |z-w|= \log|w| + o(1),
\]
uniformly for $z\in K$ as $w\to\infty$.
Thus, by the definition of $\minalp$ and the factorization of $P_\alpha$,
\[
\frac{1}{\deg \alpha} \Big(\int \log |P_{\alpha}(z)| \, d {\equi}(z) - \log^+ \minalp \Big) \longrightarrow 0 \quad \text{as } r \to \infty.
\]
Since $\log \capa(K)>0$, we may choose $r$, depending only on $K$,
sufficiently large so that 
\[
\frac{1}{\deg \alpha} \Big(\int \log |P_{\alpha}(z)| \, d {\equi}(z) - \log^+ \minalp \Big) \leq \log \capa(K).
\]
Increasing $r$, if necessary, we may also ensure that $\minalp \geq 1$.
Consequently, for every $t \in [0,1]$,
\[
\int\log|P_\alpha(z)|\,d\mu_t(\alpha,K)(z)
\geq\log \minalp \geq0.
\]
Thus both conditions of Theorem~\ref{ct:suff} are satisfied for every $t\in[0,1]$.
Hence $\mu_t (\alpha, K)$ satisfies all the inequalities in \eqref{E:infinite}.
By Result~\ref{R:APM_char}, $\mu_t(\alpha,K)$ is an arithmetic
probability measure for every $t\in[0,1]$.
\end{proof}

\begin{remark}
The methods used in the proofs of Theorems~\ref{ct:suff}~and~\ref{ct:ns} extend to a 
more general class of measures than defined in \eqref{E:conv_measure}. 
Let $K, F \subset \C$ be compact sets 
symmetric about the real axis, and let $\alpha$ be an algebraic 
number such that 
$
\conj(\alpha) \cap F = \emptyset.
$
For each $t \in [0,1]$, define the probability measure
\[
\mu_t(\alpha, K, F):= t {\equi} + \frac{1-t}{\deg \alpha} \sum_{\gamma \in \conj(\alpha)} \harmF.
\]
The conclusions of Theorems~\ref{ct:suff}~and~\ref{ct:ns} remain valid, with appropriate modifications, for 
the measures $\mu_t(\alpha,K,F)$ in place of those considered there.
\end{remark}

\medskip

\section{Arithmetic probability measures on intervals}\label{S:interval}

This section proves a generalization of Theorem~\ref{ct:interval}, allowing arbitrary convex combinations of 
the equilibrium measure and harmonic measures with respect to arbitrary real points. 
We characterize when such measures satisfy all the inequalities in \eqref{E:infinite}.
Theorem~\ref{ct:interval} is a special case of Theorem~\ref{T:intval_z}.
We begin by introducing this class of measures.
\smallskip

Let $K \subset \R$ be a compact set, and let $y_1, \dots, y_n \in \R \setminus K$. 
For each vector $\textbf{\textit{t}}= (t_0, \dots, t_n)$ satisfying $0 \leq t_i \leq 1$ and $\sum_{i=0}^n t_i=1$, 
we consider the probability measure
\begin{equation}\label{E:conv_measure_gen}
\mut(\{y_i\}, K):= t_0 {\equi} +  \sum_{i=1}^n t_i \harmyi.
\end{equation}
Before characterizing when $\mut$ satisfies all the inequalities in \eqref{E:infinite}, we establish an 
auxiliary lemma concerning the minimum of a weighted sum of logarithmic functions.

\begin{lemma}\label{L:min}
Let $[a, b] \subset \R$ be a closed interval.
Let $y_i \in \R \setminus [a, b]$ and $t_i\geq 0$ for all $1\leq i\leq n$.
Consider the function $F:[a,b] \to \R$ given by
\[
F(x)=\sum_{i=1}^{n} t_i \log |x-y_i|,
\]
Then $F$ attains its minimum on $[a,b]$ at one of the endpoints. 
\end{lemma}

\begin{proof}
Since $y_i  \in \R \setminus [a, b]$, we have
\[
\frac{d^2}{dx^2}\log|x-y_i|=-\frac{1}{(x-y_i)^2}<0
\]
on $[a,b]$. Since $t_i\geq0$, each term $t_i \log |x-y_i|$ is concave.
It follows that $F$ is concave.
A concave function on a closed interval attains its minimum at an endpoint,
which proves the claim.
\end{proof}

The following theorem provides a finite characterization of the inequalities in \eqref{E:infinite} for the measures 
$\mut$ when the underlying interval has integer endpoints. 
In particular, it reduces the full collection of inequalities in \eqref{E:infinite} to two explicit integer polynomials:

\begin{theorem}\label{T:intval_z}
Let $K=[a, b]$ be a closed interval with $a, b \in \Z$ and $a<b$, and
let $\mut:=\mut(\{y_i\}, K)$ be the measure defined in \eqref{E:conv_measure_gen}.
Then $\mut$ satisfies all the inequalities in \eqref{E:infinite}
if and only if
\[\int \log |z-a| \, d \mut(z) \geq 0 \quad {\text{and}} \quad
\int \log |z-b| \, d \mut(z) \geq 0.\]
\end{theorem}

\begin{proof}
The forward implication is immediate, since the polynomials $z-a, z-b \in\Z[z]$.
For the converse, let $Q\in \C[z]$ be a polynomial with leading coefficient $\ell_Q$ 
satisfying $|\ell_Q|\geq1$, and suppose that $Q(y_i)\neq0$ for every $1\leq i\leq n$. Write
$Q$ in its factored form
$
Q(z)=\ell_Q\prod_{j=1}^{\deg Q}(z-r_j).
$
By using Lemmas~\ref{L:jensen}~and~\ref{L:harm}, we get
\begin{align}
\int \log |Q(z)| \, d \mut(z)
= &t_0 \Big(\log |\ell_Q| + \deg Q \log \capa(K) + \sum_{j=1}^{\deg Q} g_K(r_j, \infty) \Big)   \notag \\
&+\sum_{i=1}^n t_i \Big(  \log |Q(y_i)|- \deg Q g_K(y_i, \infty) 
+ \sum_{j=1}^{\deg Q} g_K(r_j, y_i) \Big). \notag
\end{align}
Since the dependence 
on $Q$ enters only through its leading coefficient and its roots, 
we isolate the root dependent part by defining
\[
H_{\textbf{\textit{t}}}(z):= t_0 g_K(z, \infty)+ \sum_{i=1}^n t_i \big(\log |z -y_i |+ g_K(z, y_i)\big).
\]
With this notation, the preceding identity becomes
\begin{equation}\label{E:intQ}
\frac{1}{\deg Q}\int \log |Q(z)| \, d \mut(z)
= t_0 \log \capa(K) - \sum_{i=1}^n t_i g_K(y_i, \infty) + \frac{\log|\ell_Q|}{\deg Q}
+\frac{1}{\deg Q} \sum_{j=1}^{\deg Q}  H_{\textbf{\textit{t}}}(r_j). 
\end{equation}

By an argument similar to that in the proof of Lemma~\ref{L:lower_geo},
the singularities of $H_{\textbf{\textit{t}}}$ at $y_i$ are removable, so 
$H_{\textbf{\textit{t}}}$ extends harmonically to $\C\setminus[a,b]$. 
Since $[a,b]$ is regular, this extension is continuous on $\C$.
Moreover,
\[
\lim_{z \to \infty} H_{\textbf{\textit{t}}}(z)\to+\infty,
\]
independent of the parameter vector ${\textbf{\textit{t}}}$.
Hence, the minimum of $H_{\textbf{\textit{t}}}$ is attained on $[a,b]$.
For $z\in[a,b]$, the Green function terms vanish, and thus
\[
H_{\textbf{\textit{t}}}(z)= \sum_{i=1}^n t_i \log |r -y_i |.
\]
By Lemma~\ref{L:min}, the minimum of $H_{\textbf{\textit{t}}}$ on $[a,b]$,
and hence on $\C$, is attained at one of the endpoints of $[a, b]$. 
Without loss of generality, suppose that this minimum is attained at $a$. 
Note that
\[
\int \log |z-a| \, d \mut(z)
= t_0 \log \capa(K) - \sum_{i=1}^n t_i g_K(y_i, \infty) + H_{\textbf{\textit{t}}}(a).
\]
Since $H_{\textbf{\textit{t}}}(a) \leq H_{\textbf{\textit{t}}}(z)$ for all $z\in\C$, 
comparing the above expression with \eqref{E:intQ} yields
\[
\frac{1}{\deg Q}\int \log |Q(z)| \, d \mut(z)\geq \int \log |z-a| \, d \mut(z).
\]
By hypothesis, the right-hand side is nonnegative. Hence
\[
\int \log |Q(z)| \, d \mut(z) \geq 0
\] 
whenever $Q\in\C[z]$ satisfies $Q(y_i)\neq0$ for all $1\leq i\leq n$ and $|\ell_Q|\geq1$. 
In particular, this applies to any $Q\in\mathbb Z[z]$ satisfying $Q(y_i)\neq0$ for all $1\leq i\leq n$.
\smallskip

Finally, suppose $Q\in\mathbb Z[z]$ and $Q(y_i)=0$ for some $i$. Let
$I=\{i:Q(y_i)=0\}$ and write
\[
Q(z)=R(z)\prod_{i\in I}(z-y_i).
\]
Since the factors $z-y_i$ are monic, $|\ell_R|=|\ell_Q|\geq1$, and
$R(y_i)\neq0$ for every $1\leq i \leq n$. Thus, the preceding argument applies to $R$.
For each $i\in I$, since $y_i \in \R \setminus [a,b]$,
\[
|z-y_i|\geq
\begin{cases}
|z-a|,& y_i<a,\\
|z-b|,& y_i>b,
\end{cases}
\qquad z\in[a,b].
\]
Thus, by the assumed integral inequalities at the endpoints,
\[
\int\log|z-y_i|\,d\mut(z)\geq0.
\]
Combining these inequalities gives
\[
\int\log|Q(z)|\,d\mut(z) \geq \int \log|R(z)| \,d\mut(z) 
+\sum_{i\in I} \int \log|z-y_i| \,d\mut(z) 
\geq0,
\]
which completes the proof.
\end{proof}

The argument above extends verbatim to finite disjoint unions of intervals with integer endpoints, 
yielding the following corollary.

\begin{corollary}
Let $K= \cup_{j=1}^m [a_j, b_j]$ be a finite disjoint union of closed intervals with $a_j, b_j \in \Z$
and $a_j<b_j$ for $1\leq j\leq m$.
Let $\mut:=\mut(\{y_i\}, K)$ be the measure defined in \eqref{E:conv_measure_gen}.
Then $\mut$ satisfies all the inequalities in \eqref{E:infinite}
if and only if for every $1\leq j\leq m$
\[
\int \log |z-a_j| \, d \mut(z) \geq 0 \quad \text{and} \quad \int \log |z-b_j| \, d \mut (z) \geq 0.
\]
\end{corollary}

Observe that in the proof of Theorem~\ref{T:intval_z}, 
the assumption $a, b \in \Z$ is used only for necessity; the sufficiency argument holds without this 
assumption. The same observation applies to finite disjoint unions of intervals with arbitrary real endpoints, 
but we record it below only in the case of a single interval.

\begin{corollary}\label{C:realinterval}
Let $K=[a, b]$ be a closed interval with $a, b \in \R$ and $a<b$, and
let $\mut:=\mut(\{y_i\}, K)$ be the measure defined in \eqref{E:conv_measure_gen}.
If
\[\int \log |z-a| \, d \mut(z) \geq 0 \quad {\text{and}} \quad
\int \log |z-b| \, d \mut(z) \geq 0,\]
then $\mut$ satisfies all the inequalities in \eqref{E:infinite}.
\end{corollary}

\begin{remark}
The above extension to intervals with arbitrary real endpoints, together with Theorems~\ref{ct:suff}~and~\ref{ct:ns},
could be particularly useful for closed intervals of the form
\[
[-2 \sqrt q, 2 \sqrt q],
\]
where $q$ is a power of a prime number.
Such intervals naturally arise in the study of Weil numbers and Frobenius eigenvalues. 
\smallskip

For an integer interval $K=[a,b]$, Theorem~\ref{T:intval_z} reduces all the inequalities in \eqref{E:infinite}
to those corresponding to the two polynomials $z-a$ and $z-b$. It would be interesting to determine, 
for each $q$, whether an analogous finite collection of integer polynomials can be found for the interval 
$[-2 \sqrt q, 2 \sqrt q].$
More generally, one could ask for such a finite reduction for arbitrary 
intervals with nonintegral endpoints.
\end{remark}
\medskip

\section*{Acknowledgments}
\noindent{The author would like to thank Norm Levenberg for interesting discussions
regarding the paper.
This work is supported by TÜBİTAK grant ARDEB-1001/124F370.}


\medskip

\begin{thebibliography}{88888888}

\bibitem[Bil97]{Bi} Yuri Bilu, {\em Limit distribution of small points on algebraic tori},
Duke Math. J. {\bf 89} (1997), no.~3, 465--476.

\bibitem[BMQS26]{BMQS:cgemh26}
José Ignacio Burgos Gil, Ricardo Menares, Binggang Qu and Martín Sombra, {\em Closing the gap around the essential minimum of height functions with linear programming},
arXiv reference: \href{https://arxiv.org/abs/2601.18978}{\texttt{arXiv:2601.18978}}.

\bibitem[Fek23]{Fekete:udvwbgmg23}
Michael Fekete, {\em \"{U}ber die Verteilung der Wurzeln bei gewissen algebraischen Gleichungen mit ganzzahligen Koeffizienten},
Math. Z. {\bf 17} (1923), no.~1, 228--249.

\bibitem[FS55]{FekeSzego:aeicwr55}
Michael Fekete and G\'{a}bor Szeg\"{o}, {\em On algebraic equations with integral coefficients whose roots belong to a given point set},
Math. Z. {\bf 63} (1955), 158--172.

\bibitem[LL26a]{LevLon:evoft24}
Norm Levenberg and Mayuresh Londhe, {\em Around Fekete’s theorem},
Potential Anal. {\bf} 64 (2026), no. 1, Paper No. 2, 18 pp.

\bibitem[LL26b]{LL:SSS26}
Norm Levenberg and Mayuresh Londhe, {\em On the Schur--Siegel--Smyth trace problem}, 
preprint.

\bibitem[NT25]{NT:cdmcav25}
Nikolai Nadirashvili and Michael Tsfasman, {\em Complete description of measures corresponding to Abelian varieties over finite fields},
Finite Fields Appl. {\bf 101} (2025), Paper No. 102543, 12 pp.


\bibitem[OS23a]{OrlSar:ldcai23}
Bryce Joseph Orloski and Naser Talebizadeh Sardari, {\em Limiting distributions of conjugate algebraic integers},
arXiv reference: \href{https://arxiv.org/abs/2302.02872}{\texttt{arXiv:2302.02872}}.


\bibitem[OS23b]{OrlSar:qcft23}
Bryce Joseph Orloski and Naser Talebizadeh Sardari, {\em A quantitative converse of Fekete's theorem},
arXiv reference: \href{https://arxiv.org/abs/2304.10021}{\texttt{arXiv:2304.10021}}.

\bibitem[Pri11]{Pritsker:crelle11}
Igor Pritsker, {\em Distribution of algebraic numbers},
J. reine angew. Math. {\bf 657} (2011), 57--80.

\bibitem[Ran95]{ransford:ptCp95}
Thomas Ransford, {\em Potential Theory in the Complex Plane}, LMS Student Texts {\bf 25},
Cambridge University Press, Cambridge (UK), 1995.

\bibitem[Rum99]{Ru} Robert Rumely, {\em On Bilu's equidistribution theorem}, 
Spectral problems in geometry and arithmetic, Contemp. Math. {\bf 237}, AMS (1999), 159--166.


\bibitem[ST24]{SaffTotik:lpwef97}
Edward B. Saff and Vilmos Totik, {\em Logarithmic Potentials with External Fields},
Second edition, Grundlehren Math. Wiss., 316,
Springer, Cham, 2024, 2024. xvi+594 pp.

\bibitem[Ser19]{Serre19:dadvpef}
Jean-Pierre Serre, {\em Distribution asymptotique des valeurs propres des endomorphismes de Frobenius},
Ast\'erisque No. 414, S\'eminaire Bourbaki, Vol. 2017/2018. Expos\'es 1136--1150 (2019). 

\bibitem[Smi24]{Smith:aicpd24}
Alexander Smith, {\em Algebraic integers with conjugates in a prescribed distribution}, 
Ann. of Math. (2) {\bf 200} (2024), no.\,1, 71--122.

\bibitem[Tsf19]{T:serre19}
Michael Tsfasman, {\em Serre's theorem and measures corresponding to abelian varieties over finite fields},
Mosc. Math. J. {\bf 19} (2019), no.\,4, 789--806.

\end{thebibliography}
\end{document}